\documentclass[reqno]{amsart}
\usepackage[utf8]{inputenc}
\usepackage{amsmath}
\usepackage[left=2.7cm,right=2.7cm,top=3.5cm,bottom=3cm]{geometry}
\usepackage{xfrac}
\usepackage{graphicx}
\usepackage{amsfonts, amssymb}
\usepackage{tikz-cd}
\usetikzlibrary{arrows,arrows.meta}
\tikzcdset{arrow style=tikz,
	diagrams={>={Stealth[round,length=4pt,width=6pt,inset=3pt]}}}
\usepackage{amsthm}
\usepackage[T1]{fontenc}
\usepackage{xfrac}
\usepackage{color}
\usepackage[
    colorlinks = true,
    linkcolor = black,
    urlcolor  = black,
    citecolor = blue,
    bookmarksopen = true,
    backref = page
]{hyperref}

\newtheorem{Theorem}{Theorem}[section]
\newtheorem{Proposition}[Theorem]{Proposition}

\newtheorem{Lemma}[Theorem]{Lemma}

\newtheorem{Corollary}[Theorem]{Corollary}

\newtheorem{Remark}[Theorem]{Remark}
\newtheorem{innercustomgeneric}{\customgenericname}
\providecommand{\customgenericname}{}
\newcommand{\newcustomtheorem}[2]{%
	\newenvironment{#1}[1]
	{%
		\renewcommand\customgenericname{#2}%
		\renewcommand\theinnercustomgeneric{##1}%
		\innercustomgeneric
	}
	{\endinnercustomgeneric}
}

\newcustomtheorem{customthm}{Theorem}

\newcustomtheorem{customcor}{Corollary}

\begin{document}

\setcounter{tocdepth}{1}

\title{The fundamental group of surfaces parametrizing cuboids}
\author{Benjamin Enriquez}
\address{Institut de Recherche Math\'ematique Avanc\'ee, Universit\'e de Strasbourg, France}
\email{b.enriquez@math.unistra.fr}

\author{David Jarossay}
\address{De Vinci Higher Education, De Vinci Research Center, Paris, France}
\email{david.jarossay@devinci.fr}

\author{Francesco Maria Saettone}
\address{ Department of Mathematics, Weizmann Institute of Science, Israel}
\email{francesco.saettone@weizmann.ac.il}

\author{ Yotam Svoray}
\address{Department of Mathematics, University of Utah, Utah}
\email{svoray@math.utah.edu}

\begin{abstract}
We prove that an irreducible projective complete intersection of dimension at
least two with isolated singularities has trivial fundamental group. As an
application, the surface $\Upsilon$ parametrizing cuboids and its minimal
resolution of singularities are simply connected. By an independent argument we also show that the surface $V$ parametrizing face cuboids and its resolution are simply connected as well. We then introduce two smooth open subvarieties $S_{1}$ and $S_{2}$ of the surface parametrizing face cuboids, show that each has fundamental group isomorphic to $\mathbb F_{3}\ltimes \mathbb Z^{2}$, and prove that their Malcev completions reduce to the free pro-unipotent group on three generators. In an appendix we treat the corresponding real loci, whose
fundamental groups, in contrast, are far from trivial.
\end{abstract}

\maketitle

\tableofcontents

\section{Introduction}

\subsection{Cuboids and face cuboids}
 
A \emph{cuboid} is a rectangular box. It is described by seven quantities: the
three edge lengths $A,B,C$, the three face diagonals $X,Y,Z$, and the space
diagonal $U$. Naturally, these are not independent, but satisfy
the four relations
\begin{equation}\label{eq:cuboids}
\begin{aligned}
A^{2}+B^{2}-Z^{2} &= 0,\\
B^{2}+C^{2}-X^{2} &= 0,\\
C^{2}+A^{2}-Y^{2} &= 0,\\
A^{2}+X^{2}-U^{2} &= 0.
\end{aligned}
\end{equation}
A \emph{perfect cuboid} is one all of whose seven lengths are integers. Whether
a perfect cuboid exists is a venerable open problem, going back to Euler, which
has so far resisted several investigations: no
example is known, and none has been ruled out.
 
It is natural to approach this question geometrically, by assembling all cuboids
into a single algebraic variety. Following van~Luijk~\cite{VL}, the
\emph{surface parametrizing cuboids} (also known as the \emph{box variety}
in~\cite{FSM}) is the closed subvariety
\[
\Upsilon\subset \mathbb P^{6}_{\mathbb C}
\]
cut out by the equations~\eqref{eq:cuboids} in the homogeneous coordinates
$A,B,C,X,Y,Z,U$. Since these equations are homogeneous, every point of
$\Upsilon(\mathbb Q)$ acquires integral homogeneous coordinates after clearing
denominators, so that
$\Upsilon(\mathbb Z)=\Upsilon(\mathbb Q)$,
and the integral points of $\Upsilon$ are precisely the perfect cuboids. In
particular, a perfect cuboid exists if and only if
$\Upsilon(\mathbb Q)\neq\emptyset$.
 
As shown by van~Luijk, $\Upsilon$ is a complete intersection with exactly $48$
complex singular points, all ordinary double points. We write
$\mathrm{Sing}(\Upsilon)$ for this singular locus and
$\Upsilon^{\mathrm{sm}}:=\Upsilon\smallsetminus \mathrm{Sing}(\Upsilon)$ for the
smooth locus, and we denote by $\widetilde{\Upsilon}$ the minimal resolution of
singularities, a smooth projective surface.
 
Relaxing the integrality condition slightly leads to a second, closely related
surface. A \emph{face cuboid} is a box whose lengths are all integers except
possibly one of the three face diagonals $X,Y,Z$. The associated moduli space is
the \emph{surface parametrizing face cuboids}, the subvariety
\[
V\subset \mathbb P^{5}_{\mathbb Q}
\]
defined in coordinates $A,B,C,X,Y,U$ by
\begin{equation}\label{eq:face_cuboids}
\begin{aligned}
A^{2}+C^{2}-Y^{2} &= 0,\\
B^{2}+C^{2}-X^{2} &= 0,\\
A^{2}+X^{2}-U^{2} &= 0,
\end{aligned}
\end{equation}
whose integral points are the face cuboids. In contrast to perfect cuboids, face
cuboids exist in abundance (see~\cite{VL} and the references therein). First
studied by Beukers and van~Geemen~\cite{BvG}, the surface $V$ is again a complete
intersection, now with $16$ ordinary double points~\cite[Prop.~4.1, p.~51]{VL}.
It carries a more transparent structure: it is the quotient of a product
of two elliptic curves by an involution, and its minimal resolution
$\widetilde V$ is the associated Kummer surface.

\subsection{Motivation}
 
Our interest in these surfaces stems from the prospect of studying the arithmetic
of cuboids through fundamental groups. The Chabauty--Kim method, introduced
in~\cite{CK}, controls the rational points of a hyperbolic curve by means of its
unipotent fundamental group. A fully developed analogue for surfaces is not yet available (with the  partial exception of~\cite{JDC}) but the same philosophy suggests that the
unipotent fundamental groups attached to $\Upsilon$, to $V$, and to suitable
open subvarieties should carry meaningful arithmetic content. The present note
is a first, initial step in this direction.

\subsection{Main results}
 
It is a classical result that a \emph{smooth} projective irreducible complete intersection of complex
dimension at least $2$ is simply connected; see~\cite[Ch.~IX, \S4.1]{ShafarevichBAG2}.
In the present work we extend this to the isolated singularities setting, which applies well beyond the cuboid surfaces and may be
of independent interest.
 
\begin{customthm}{A}\label{thm:A}
Let $n\geq 2$ and $m>n$, and let $X\subset \mathbb P^{m}_{\mathbb C}$ be an
irreducible projective complete intersection of dimension $n$ with isolated
singularities. Then $X(\mathbb C)$  has trivial fundamental group.
\end{customthm}
 
The hypothesis $n\geq 2$ is essential, since a singular complete intersection
curve may well have nontrivial fundamental group. Our proof, inspired
by~\cite[Theorem~2.1]{dimca}, proceeds by analysing a general hyperplane section
$W=X\cap H$. By a result of Dimca~\cite{dimca}, refining earlier work of
Looijenga~\cite{loo}, the complement $X(\mathbb C)\smallsetminus W(\mathbb C)$
has the homotopy type of a bouquet of $n$-spheres, hence is simply connected when
$n\geq 2$. A Seifert--van Kampen argument, comparing this complement with a
tubular neighborhood of $W$, then forces $\pi_1(X(\mathbb C))$ to vanish.
 
The surface $\Upsilon$ is an irreducible projective complete intersection with
only ordinary double points, so Theorem~\ref{thm:A} applies and yields
$\pi_1(\Upsilon(\mathbb C))=0$; a companion argument shows that $\widetilde{\Upsilon}(\mathbb C)$ is simply connected as well. The face cuboid surface $V$ admits a quite different and more elementary treatment: from its
description as the quotient of a product of two elliptic curves by an involution,
the triviality of $\pi_1(V(\mathbb C))$ follows from a classical theorem of
Armstrong~\cite{arm} on fundamental groups of orbit spaces, and that of
$\pi_1(\widetilde V(\mathbb C))$ from Spanier's computation of the homology of
Kummer manifolds~\cite{spanier}. We thus obtain the following.

\begin{customcor}{A}\label{cor:B}
The complex analytic spaces $\Upsilon(\mathbb C)$, $\widetilde{\Upsilon}(\mathbb C)$,
$V(\mathbb C)$, and $\widetilde V(\mathbb C)$ have trivial fundamental groups.
\end{customcor}

The picture becomes richer once a divisor is removed. Inside $\widetilde V$ we
single out two smooth open surfaces $S_{1},S_{2}$, each the complement of a normal
crossings divisor, which fiber over a four-punctured sphere with fibers an
elliptic curve. The long exact homotopy sequence of these fibrations, made
explicit by a section, yields a semidirect product decomposition.
 
\begin{customthm}{B}\label{thm:C}
Let $\mathbb F_{3}$ denote the free group on three generators. Then both
$\pi_{1}\big(S_{1}(\mathbb C)\big)$ and $\pi_{1}\big(S_{2}(\mathbb C)\big)$ are
isomorphic to $\mathbb F_{3}\ltimes \mathbb Z^{2}$.
\end{customthm}
 
We then pass to the Malcev, or unipotent, completions of these
groups---the Betti realization of the unipotent fundamental group, and the
natural input for a Chabauty--Kim analysis. The monodromy of the fibration acts
on the elliptic fiber by $-1$, and we show that this collapses the entire fiber
upon completion: the $\mathbb Z^{2}$ disappears, leaving only the contribution of
the base. In symbols,
\[
\pi_{1}\big(S_{i}(\mathbb C)\big)^{\mathrm{un}}\;\simeq\;\mathbb F_{3}^{\mathrm{un}}
\qquad 
\text{for}
\quad 
i=1,2
\]
the free pro-unipotent group on three generators.
 
Finally, an appendix examines the real loci $\Upsilon(\mathbb R)$,
$\widetilde{\Upsilon}(\mathbb R)$, $V(\mathbb R)$, and $\widetilde V(\mathbb R)$.
In sharp contrast with the complex case, these are far from simply connected: we
describe each up to homeomorphism, as a connected sum of non-orientable surfaces
with the singular points contributing free factors to the fundamental group, and
we record a formula, following \cite{bruce}, expressing the relevant Euler
characteristics through real Milnor numbers.

\subsubsection*{Acknowledgments} 
D.J. thanks Ishai Dan-Cohen\footnote{In turn, Ishai Dan-Cohen wishes to thank Ambrus P\'al for introducing \textit{him} to this problem. To the best of our knowledge, the idea to approach the arithmetic of cuboids via fundamental group techniques originates from him.} for introducing him to this problem and Netan Dogra for correspondence. Y.S. thanks Brennan Richardson for his help with computations over $\mathbb{R}$. We also thank an anonymous referee for useful comments.

F.M.S. was supported by the ERC, SharpOS, 101087910, by the ISF grants  2067/23 and 1963/20, and the BSF grant 2018250.

\section{Triviality of fundamental groups}

\subsection{Computation of the fundamental groups}

\begin{Theorem}\label{p:ci-variety-pi1}
Let $n\geq 2$, let $m>n$, and let
$X\subset \mathbb P^m_{\mathbb C}$ be an irreducible projective complete
intersection of dimension $n$ with isolated singularities. Then
\[
\pi_1\big(X(\mathbb C)\big)=0.
\]
\end{Theorem}

\begin{proof}
Let $H\subset \mathbb P^m_{\mathbb C}$ be a general hyperplane. We first
explain what general means.
By Bertini's theorem \cite[Lemma~33.47.3]{stacks-project}, applied to the
smooth quasi-projective variety
\[
X^{\mathrm{sm}}\subset \mathbb P^m_{\mathbb C},
\]
the hyperplane $H$ meets $X^{\mathrm{sm}}$ transversally. Hence
$X\cap H$ is smooth at every point of $X^{\mathrm{sm}}\cap H$.

Moreover, since $X$ has only finitely many isolated singular points, we may
also choose $H$ so that
\[
H\cap \mathrm{Sing}(X)=\emptyset .
\]
Indeed, the set of hyperplanes containing a fixed point of $\mathrm{Sing}(X)$
is a proper closed subset of the dual projective space
\[
(\mathbb P^m_{\mathbb C})^\vee
\]
and there are only finitely many such points. Thus the hyperplanes avoiding
$\mathrm{Sing}(X)$ form a nonempty Zariski open subset of
$(\mathbb P^m_{\mathbb C})^\vee$.

We next verify the hypotheses of Bertini irreducibility
\cite[Tag~0G4F]{stacks-project}. Set
\[
\mathcal L:=\mathcal O_X(1),\qquad
V:=H^0(\mathbb P^m_{\mathbb C},\mathcal O(1)),
\]
and let
\[
V\rightarrow H^0(X,\mathcal L)
\]
be the restriction map.

First, $X$ is of finite type over $\mathbb C$, since it is closed in
$\mathbb P^m_{\mathbb C}$. Moreover, $X$ is irreducible by assumption; since
$\mathbb C$ is algebraically closed, it is geometrically irreducible.

We next verify the third hypothesis. Choose a smooth point $p\in X$, which
exists because $\mathrm{Sing}(X)$ is finite and $\dim X=n$. Let $\mathcal I_p$ denote the ideal sheaf of the point $p$ in $\mathbb P^m_{\mathbb C}$, and denote $\mathcal I_p(1):=\mathcal I_p\otimes \mathcal O_{\mathbb P^m_{\mathbb C}}(1)$. 

Let $\mathfrak m_{\mathbb P,p}\subset \mathcal O_{\mathbb P^m_{\mathbb C},p}$ and
$\mathfrak m_{X,p}\subset \mathcal O_{X,p}$
be the maximal ideals at $p$. We claim that the natural map
\[
H^0\big(\mathbb P^m_{\mathbb C},\mathcal I_p(1)\big)
\rightarrow
\mathfrak m_{X,p}/\mathfrak m_{X,p}^2
\]
is surjective. This map is defined as follows. A section of
$\mathcal I_p(1)$ is a hyperplane section vanishing at $p$. After choosing
coordinates with $p=[1:0:\cdots:0]$, such a section is given by a linear form
\[
\ell=a_1x_1+\cdots+a_mx_m .
\]
On the affine chart $x_0\neq 0$, its local equation is
$\frac{\ell}{x_0}=a_1t_1+\cdots+a_mt_m$,for $t_i:=x_i/x_0$. Restricting this local equation to $X$ gives an element of
$\mathfrak m_{X,p}$, and the map sends $\ell\mapsto
\dfrac{\ell}{x_0}|_X\in \mathfrak m_{X,p}/\mathfrak m_{X,p}^2$.
The classes of $t_1,\ldots,t_m$ generate
$\mathfrak m_{\mathbb P,p}/\mathfrak m_{\mathbb P,p}^2$.
Since $X\hookrightarrow \mathbb P^m_{\mathbb C}$ is a closed immersion, the
local homomorphism
$\mathcal O_{\mathbb P^m,p}\twoheadrightarrow \mathcal O_{X,p}$
is surjective, and therefore induces a surjection
\[
\mathfrak m_{\mathbb P,p}/\mathfrak m_{\mathbb P,p}^2
\twoheadrightarrow
\mathfrak m_{X,p}/\mathfrak m_{X,p}^2\ .
\]
Hence the classes of the restricted hyperplane equations generate
$\mathfrak m_{X,p}/\mathfrak m_{X,p}^2$. 
Hence we may choose two hyperplanes $H_2,H_3$ through $p$
whose local equations restrict to two linearly independent elements of
$\mathfrak m_{X,p}/\mathfrak m_{X,p}^2$.
Since $p$ is a smooth point of $X$, the local ring $\mathcal O_{X,p}$ is regular
of dimension $n$. Thus these two local equations form part of a regular system
of parameters of $\mathcal O_{X,p}$. It follows that the local codimension of
\[
X\cap H_2\cap H_3
\]
at $p$ is $2$. Let $Z$ be an irreducible component of this intersection passing
through $p$. Then
\[
\operatorname{codim}_X(Z)=2.
\]
Choose a hyperplane $H_1$ not containing $p$; then $Z\not\subset H_1$.

Finally, the complete linear system of hyperplanes has empty base locus on $X$,
i.e.,
\[
\bigcap_{H\subset\mathbb P^m_{\mathbb C}}(X\cap H)=\emptyset.
\]
Thus the fourth hypothesis of \cite[Tag~0G4F]{stacks-project} is also satisfied.
It follows that $X\cap H$ is irreducible for a general hyperplane $H$.

We now choose a hyperplane $H$ satisfying the following three open conditions:
\[
H \text{ meets } X^{\mathrm{sm}} \text{ transversally},\qquad
H\cap \mathrm{Sing}(X)=\emptyset,\qquad
X\cap H \text{ is irreducible}.
\]
This is possible because the first two conditions define nonempty Zariski open subsets of $(\mathbb P^m_{\mathbb C})^\vee$, and the third is the nonempty
Zariski open condition obtained from Bertini irreducibility. 
\\
We now set
\[
W:=X\cap H.
\]
Then $W$ is smooth and irreducible, hence path-connected. Set
\[
U:=X(\mathbb C)\smallsetminus W(\mathbb C),
\]
and let $N$ be a sufficiently small open tubular neighborhood of
$W(\mathbb C)$ in $X(\mathbb C)$, contained in
$X^{\mathrm{sm}}(\mathbb C)$. This is possible because $W$ is compact and
disjoint from $\mathrm{Sing}(X)$. Since $N$ deformation retracts onto the path-connected space $W(\mathbb C)$, it is path-connected. Indeed, irreducible topological spaces are connected \cite[Tag~004U]{stacks-project}, and connectedness is preserved under analytification. Since $W$ is
smooth over $\mathbb C$, we have that $W(\mathbb C)$ is a complex manifold, hence locally
path-connected. Therefore connectedness implies path-connectedness.

Now, $X$ is a projective complete intersection of complex dimension $n$
with isolated singularities, and $W=X\cap H$ is smooth. Therefore, by
\cite[\S2, pp.~324--325]{dimca}, the space $U$ has the homotopy type of a
bouquet of $n$-spheres. Since $n\ge 2$, it follows that  $U$ is path-connected and
\[
\pi_1(U)=0.
\]

We now show that $U\cap N$ is path-connected. Since $N$ is a tubular
neighborhood of the smooth divisor
\[
W(\mathbb C)\subset X^{\mathrm{sm}}(\mathbb C),
\]
we have
\[
U\cap N=N\smallsetminus W(\mathbb C).
\]
The space $N\smallsetminus W(\mathbb C)$ deformation retracts onto the
$S^1$-bundle of the normal line bundle $\mathcal N_{W/X^{\mathrm{sm}}}$:
\[
S^1\rightarrow S(\mathcal N_{W/X^{\mathrm{sm}}})\rightarrow W(\mathbb C).
\]
Since both $S^1$ and $W(\mathbb C)$ are path-connected, the total space
$S(\mathcal N_{W/X^{\mathrm{sm}}})$, and therefore $U\cap N$, is
path-connected.

Fix a base point $x_0\in U\cap N$, which we omit from the notation. Let
\[
j_1\colon \pi_1(U)\rightarrow \pi_1(X(\mathbb C))
\qquad
\text{and}
\qquad
j_2\colon \pi_1(N)\rightarrow \pi_1(X(\mathbb C))
\]
be induced by the inclusion maps. Since $U$, $N$, and $U\cap N$ are
path-connected, Seifert--van Kampen gives the following pushout diagram:
\begin{equation}\label{seifert-general-ci}
\begin{tikzcd}
    \pi_1(U\cap N)\arrow{r}{i_1} \arrow{d}{i_2}
    & \pi_1(U) \arrow{d}{j_1} \\
    \pi_1(N) \arrow{r}{j_2}
    & \pi_1(X(\mathbb C))
    \simeq \pi_1(N)*_{\pi_1(U\cap N)}\pi_1(U).
\end{tikzcd}
\end{equation}

It remains to show that
\[
i_2\colon \pi_1(U\cap N)\rightarrow \pi_1(N)
\]
is surjective. Since $U\cap N$ deformation retracts onto the $S^1$-bundle
$S(\mathcal N_{W/X^{\mathrm{sm}}})$, the fibration
\[
S^1\rightarrow S(\mathcal N_{W/X^{\mathrm{sm}}})\rightarrow W(\mathbb C)
\]
gives, from the long exact sequence of homotopy groups,
\[
\pi_1(S(\mathcal N_{W/X^{\mathrm{sm}}}))\rightarrow \pi_1(W(\mathbb C))
\rightarrow \pi_0(S^1)=0.
\]
Thus the map
\[
\pi_1(S(\mathcal N_{W/X^{\mathrm{sm}}}))\rightarrow \pi_1(W(\mathbb C))
\]
is surjective. Using the deformation retractions
\[
U\cap N\simeq S(\mathcal N_{W/X^{\mathrm{sm}}})
\qquad\text{and}\qquad
N\simeq W(\mathbb C),
\]
and this yields the surjectivity of $i_2$.

Since $\pi_1(U)=0$, the pushout description \eqref{seifert-general-ci}
becomes
\[
\pi_1(X(\mathbb C))
\simeq
\pi_1(N)*_{\pi_1(U\cap N)}\{0\}.
\]
By definition of the amalgamated product, this is obtained from $\pi_1(N)$ by
imposing the relations
\[
i_2(\gamma)=i_1(\gamma)
\]
for all $\gamma\in\pi_1(U\cap N)$. But $i_1(\gamma)=1$, since
$\pi_1(U)=0$. Hence
\[
\pi_1(X(\mathbb C))
\simeq
\sfrac{\pi_1(N)}
{i_2\bigl(\pi_1(U\cap N)\bigr)^{\mathcal N}},
\]
where $i_2(\pi_1(U\cap N))^{\mathcal N}$ denotes the normal closure of
$i_2(\pi_1(U\cap N))$ in $\pi_1(N)$.
Since $i_2$ is surjective, we have
\[
i_2\bigl(\pi_1(U\cap N)\bigr)=\pi_1(N).
\]
Therefore its normal closure is all of $\pi_1(N)$, and so
$\pi_1(X(\mathbb C))=0$.
\end{proof}

\begin{Corollary}
    The fundamental group of $\Upsilon(\mathbb C)$ is trivial.
\end{Corollary}

\begin{proof}
By \cite[Lemma~3.2.1]{VL}, applied with $K=\mathbb C$, the ideal defining $\Upsilon_{\mathbb C}$ is prime. Hence $\Upsilon_{\mathbb C}$ is irreducible.
Moreover, $\Upsilon_{\mathbb C}\subset \mathbb P^6_{\mathbb C}$ is a projective complete intersection surface. By \cite[Lemma~3.2.9]{VL}, its singularities are ordinary double points; in particular, they are isolated. Therefore Theorem~\ref{p:ci-variety-pi1}, applied with $n=2$ and $m=6$, gives
$\pi_1\big(\Upsilon(\mathbb C)\big)=0$.
\end{proof}

\begin{Corollary}\label{t:Upsilon-tilde}
The fundamental group of $\widetilde{\Upsilon}(\mathbb C)$ is trivial.
\end{Corollary}

\begin{proof}
By \cite[Lemma~3.2.9]{VL}, the singularities of $\Upsilon_{\mathbb C}$ are
ordinary double points. In particular, they are quotient singularities. By
\cite[Definition~3.2.10 and Corollary~3.2.11]{VL},
$\widetilde{\Upsilon}_{\mathbb C}$ is the blow-up of $\Upsilon_{\mathbb C}$ at
its $48$ singular points, and this blow-up is nonsingular.
Therefore 
\cite[Theorem~7.8.1]{KollarShafarevich} gives
\[
\pi_1\big(\widetilde{\Upsilon}(\mathbb C)\big)
\simeq
\pi_1\big(\Upsilon(\mathbb C)\big)
\]
where the right-hand side is trivial by Theorem~\ref{p:ci-variety-pi1}.
\end{proof}

\begin{Remark}
{\em van Luijk showed (\cite{VL}, Proposition 3.2.19) that the Hodge diamond of $\widetilde{\Upsilon}(\mathbb{C})$ is of the form 
$$ \begin{array}{c} 1 
\\ q \text{ }\text{ }\text{ }\text{ }\text{ }q 
\\ 7+q\text{ }\text{ }\text{ } h^{1,1}\text{ }\text{ }\text{ } 7+q 
\\ q \text{ }\text{ }\text{ }\text{ }\text{ } q 
\\ 1 
\end{array} $$
He then conjectured that $q=0$ (\cite{VL}, Bluff 1). Later, Stoll and Testa proved that $q=0$ and $h^{1,1}= 64$ in \cite[p.4]{ST}. Here, the fact that $H_1(\widetilde{\Upsilon}(\mathbb{C}))=0$ implies in particular that  $h^{0,1}=h^{1,0}=0$, which recovers the fact that $q=0$.}
\end{Remark}

\section{Fundamental groups of two surfaces}

It is somehow easier to work with $V$ than with $\Upsilon$. Since we have a surjective map $\Upsilon(K) \rightarrow V(\mathbb{Q})$ where $K$ is the extension of $\mathbb{Q}$ generated by square roots of integers, in order to study perfect cuboids it would be sufficient to understand the rational points of $V$.

\subsection{A description of $V$ via elliptic curves}

We review basic properties of $V$ following \cite[pp.52-53]{VL}.
\\Consider the following elliptic curve
\[
E:\;\; y^2z=x^3-4xz^2
\]
and the torsion point $T = [0:0:1]$. 

We have that $E \times E$ is an abelian surface with two automorphisms of order $2$, namely $$\iota\colon (P,Q) \mapsto (-P,-Q),\;\;\text{and}\;\;\gamma\colon (P,Q) \mapsto (P + T, Q + T).$$

The map $\Phi : E \times E \rightarrow V$ given by 
$$ \left\{ \begin{array}{l} A = y_{1}^{2}y_{2}^{2} - 16 x_{1}^{2}x_{2}^{2} 
\\ B = 4 ( y_{1}^{2} x_{2}^{2} - y_{2}^{2} x_{1}^{2})
\\ C = 8 x_{1}x_{2}y_{1}y_{2}
\\ X = 4( y_{1}^{2}x_{2}^{2}+y_{2}^{2}x_{1}^{2})
\\ Y = y_{1}^{2}y_{2}^{2} + 16 x_{1}^{2}x_{2}^{2} 
\\ U = (y_{1}^{2}+8 x_{1}z_{1}) (y_{2}^{2}+8 x_{2}z_{2})
\end{array} 
\right. $$
induces an isomorphism 
$$ (E \times E) / \langle \iota,\gamma\rangle \simeq V.$$
 
Let also the elliptic curve $E' = E / \langle \tau \rangle$ where $\tau$ is the involution $P \mapsto P+ T$. We have
$$ E':\;\; y^2z=x^3+xz^2. $$
By the complex uniformization, $E$ and $E'$ correspond to lattices $\Lambda_1$ and $\Lambda_2$, respectively. Define $\tau' = \mathrm{id} \times \tau$ as the automorphism of $E \times E$, so that we have $E \times E' = (E\times E)/ \langle \tau' \rangle$

Let $\alpha\colon (P,Q) \mapsto (P+Q,Q)$ the automorphism of $E\times E$. We have $\alpha^{-1} \langle \iota,\gamma\rangle \alpha = \langle \tau',\iota \rangle$. Hence 
$$ V \simeq (E \times E) / \langle \tau',\iota \rangle. $$

Let $\mathrm{inv}\colon(P,Q) \mapsto (-P,-Q)$ the automorphism of $E \times E'$. We have 
\begin{equation*}
\begin{aligned}
&V \simeq (E \times E')  / \langle \iota,\gamma \rangle \simeq (E\times E')/ \langle\mathrm{inv}\rangle.
\end{aligned}
\end{equation*}
In particular, $V$ is the singular Kummer quotient of $E\times E'$, and its minimal resolution of singularities $\widetilde{V}$ its Kummer surface. The 16 singular points of $V$ correspond to the $2$-torsion points of $V$, i.e., the elements of $E[2] \times E'[2]$.

\begin{Proposition}\label{prop V}
    The fundamental groups of  $\pi_{1}(V(\mathbb{C}))$ and $\pi_1(\widetilde{V}(\mathbb{C}))$ are both trivial.
\end{Proposition} 
\begin{proof} 
By the above description of $V$ we immediately have
\[
V(\mathbb{C}) \simeq \mathbb{C}^{2}/G,
\]
where $\Lambda:=\Lambda_1\times\Lambda_2$ and $G:=\Lambda\rtimes \{\pm 1\}$ 
acts on $\mathbb{C}^{2}$ by $(\lambda,\varepsilon)\cdot z=\varepsilon z+\lambda$,  for $\lambda\in\Lambda$, and $\varepsilon\in\{\pm1\}$.
We now apply Armstrong's theorem \cite{arm}, which yields
\[
\pi_{1}\big(V(\mathbb{C})\big)\cong G/N
\]
where $N$ is the normal subgroup of $G$ generated by the elements of $G$ having a fixed point on $\mathbb{C}^{2}$. For every $\lambda\in\Lambda$, the element $(\lambda,-1)\in G$ has a fixed point: indeed, $(\lambda,-1)\cdot z=z$ is equivalent to
$2z=\lambda$, which has the solution $z=\lambda/2\in\mathbb{C}^{2}$. Hence
$(\lambda,-1)\in N$ for all $\lambda\in\Lambda$.
In particular, $(0,-1)\in N$. Therefore, for every $\lambda\in\Lambda$, we have $(\lambda,1)=(\lambda,-1)(0,-1)\in N$.
Thus all elements $(\lambda,1)$ and $(\lambda,-1)$ belong to $N$, which implies that  $N=G$.Hence $\pi_1(V(\mathbb{C}))=0$.

Lastly,  the triviality of $\pi_1(\widetilde{V}(\mathbb{C}))$ follows directly from \cite{spanier}.
\end{proof}

\subsubsection{Two open smooth subvarieties of $V$}

We now define the following complement of a normal crossings divisor inside a projective smooth variety
\begin{equation}
\begin{aligned}
& S_{1}= \big(E \smallsetminus E[2]\big) \times E' / \mathrm{inv}
 \\
& S_{2}= E \times \big( E' \smallsetminus E'[2]\big) / \mathrm{inv}.
\end{aligned}
\end{equation}

The open surfaces $S_1$ and $S_2$ identify with the corresponding open subsets of the resolution $\widetilde V$.
More precisely, if $\widetilde V \to V$ is the minimal resolution, then the preimage of $S_i\subset V$
is isomorphic to $S_i$ and can be written as $
S_i \;=\; \widetilde V \smallsetminus D_i$,
where $D_i$ is a normal crossings divisor.

Indeed, the boundary $D_i$ is obtained as follows: one removes from $V$ the images of the relevant components
coming from $E[2]$ (respectively $E'[2]$), and then resolves the resulting singularities by blowing up finitely many
points lying on the corresponding elliptic curves. Each blow-up replaces such a point by an exceptional $\mathbb{P}^1$ meeting the strict
transform of the curve transversely. Consequently, the total boundary divisor has only normal crossings.

\subsection{Fundamental groups}

We now compute the topological fundamental groups of the smooth open surfaces $S_1$ and $S_2$.

\begin{Theorem}\label{t:fund_group_S_1}
The topological fundamental groups  $\pi_{1}(S_{1})$ and 
 $\pi_{1}(S_{2})$ are both isomorphic to $\mathbb F_{3} \ltimes \mathbb{Z}^{2}$.
\end{Theorem}

\begin{proof} We present the proof for the $S_2$ case and the proof for $S_1$ is virtually identical. We  denote by $\mathrm{inv}$ the map $P \mapsto -P$ on $E'$. We have a fibration as follows :
\begin{equation} \label{eq:fib}  
p\colon S_{2} = (E \times ( E' - E'[2])) / \mathrm{inv} \rightarrow ( E' - E'[2]) / \mathrm{inv}\ ,\qquad (P,Q) \mapsto Q.
\end{equation}
The base space of (\ref{eq:fib}) is isomorphic to $\mathbb{P}^{1} \smallsetminus \{0,1,\lambda ,\infty\}$ where $\lambda \in \mathbb{C}\smallsetminus\{0,1\}$. Indeed, we have an isomorphism $E'(\mathbb{C}) / \mathrm{inv} \simeq \mathbb{P}^{1}(\mathbb{C})$. Whence, since $E'[2]$ is the set of fixed points of $\mathrm{inv}$, we have an isomorphism of the following form with $\lambda \in \mathbb{C}\smallsetminus\{0,1\}$ :
$$( E' - E'[2]) / \mathrm{inv} \simeq \mathbb{P}^{1} \smallsetminus \{0,1,\lambda ,\infty\}.$$

The fiber of (\ref{eq:fib}) is isomorphic to $E$. Indeed, the preimage of a class $\pm Q_{0}$ is the set of classes of elements of the form $(P,\pm Q_{0})$, and we have $\mathrm{inv}(P,Q_{0})$ and $(-P,-Q_{0})$.

Thus, the fibration (\ref{eq:fib}) induces the following long exact sequence of homotopy
\begin{multline*}
\cdots \rightarrow \pi_2(\mathbb{P}^1\smallsetminus\{0,1,\lambda,\infty\})\rightarrow \pi_1(E(\mathbb{C}))\rightarrow \pi_1(S_{2})\rightarrow
\pi_1(\mathbb{P}^1\smallsetminus\{0,1,\lambda,\infty\})\rightarrow \pi_{0}(E(\mathbb{C})) \rightarrow \cdots
\end{multline*}

Moreover, it is well-known that
\[
\pi_{2}(\mathbb{P}^{1}(\mathbb{C}) \smallsetminus \{0,1,\lambda,\infty\})=1,\;\;\pi_{1}(E(\mathbb{C})) = \mathbb{Z}^{2},\;\;\pi_{1}(\mathbb{P}^{1}(\mathbb{C}) \smallsetminus \{0,1,\lambda,\infty\})= \mathbb F_{3}.
\]
As $E(\mathbb C)$ is a torus, we immediately have $\pi_{0}(E(\mathbb{C}))=1$.
Thus, the long exact sequence yields the following short exact sequence:
\begin{equation} \label{eq:short-exact-sequence}
    1 \rightarrow \mathbb{Z}^{2}  \rightarrow\pi_{1}(S_{2})\rightarrow \mathbb F_3 \rightarrow 1.
\end{equation}

Finally, the fibration has a section $[Q] \mapsto [(0, Q)]$, which is well-defined as $\mathrm{inv}((0,Q))=(0,-Q)$. Thus, the short exact sequence (\ref{eq:short-exact-sequence}) is split.
\end{proof}

\begin{Remark}{\em 
    The above exact sequence \eqref{eq:short-exact-sequence} is split anyway since its last term in a free group, but here we are giving an explicit splitting.}
\end{Remark}

\subsection{Malcev completion}

For an introduction to Malcev completion, we refer to \cite{hain}; see also
\cite{zehavi} for a recent use of relative Malcev completion in
Chabauty--Kim theory. The Malcev completion of the topological fundamental
group gives the Betti realization of the unipotent fundamental group. We will
show that, for the two surfaces $S_1$ and $S_2$, the Malcev completion kills
the $\mathbb Z^2$-fiber and retains only the contribution coming from the
base of the fibration.

We recall that if $\Gamma$ is a discrete group, its Malcev completion over
$\mathbb Q$ is a pro-unipotent algebraic group $\Gamma^{\mathrm{un}}$,
together with a homomorphism
\[
\iota_\Gamma\colon \Gamma\rightarrow \Gamma^{\mathrm{un}}(\mathbb Q),
\]
satisfying the following universal property: for every pro-unipotent
algebraic group $U$ over $\mathbb Q$, and every group homomorphism
$f\colon \Gamma\rightarrow U(\mathbb Q)$, there exists a unique morphism
of pro-unipotent algebraic groups
\[
\widetilde f\colon \Gamma^{\mathrm{un}}\rightarrow U
\]
such that $f=\widetilde f\circ \iota_\Gamma$. This is the unipotent case
of relative completion, obtained by taking the reductive quotient to be
trivial; see \cite[\S3]{hain-modular}.

Let us first record the monodromy action. We do this for $S_2$; the case of
$S_1$ is identical, with $E$ and $E'$ interchanged. Let
\[
B:=(E'\smallsetminus E'[2])/\{\pm 1\}
\simeq
\mathbb P^1\smallsetminus\{0,1,\lambda,\infty\},
\]
and fix $b=[Q_0]\in B$. Via
$P\mapsto [(P,Q_0)]$
we identify the fiber $p^{-1}(b)$ of \eqref{eq:fib} with $E$. A small
loop in $B$ around any puncture has non-trivial monodromy for the double
cover
\[
E'\smallsetminus E'[2]\rightarrow B;
\]
hence it sends $Q_0$ to $-Q_0$. Therefore it sends $[(P,Q_0)]\mapsto
[(P,-Q_0)]=
[(-P,Q_0)]$.
Thus each standard generator of
$\pi_1(B)\simeq \mathbb F_3$
acts on
$\pi_1(E)\simeq \mathbb Z^2$
by multiplication by $-1$.

Consequently, for $i=1,2$, using the splitting from
Theorem~\ref{t:fund_group_S_1}, we may write
\[
\pi_1(S_i)\simeq \mathbb Z^2\rtimes \mathbb F_3
\]
where the normal subgroup is $\mathbb Z^2$. If $x,y$ generate this
normal subgroup and $g_1,g_2,g_3$ are free generators of $\mathbb F_3$,
then
\begin{equation}\label{e:monodromy-relations}
g_jxg_j^{-1}=x^{-1}
\qquad
\text{and}
\qquad
g_jyg_j^{-1}=y^{-1}
\end{equation}
for $j=1,2,3$.

\begin{Proposition}\label{p:malcev}
The Malcev completions of $\pi_1(S_1)$ and $\pi_1(S_2)$ over
$\mathbb Q$ are isomorphic to the pro-unipotent completion of the free group in three
generators:
\[
\pi_1(S_1)^{\mathrm{un},B}
\simeq
\mathbb F_3^{\mathrm{un}}
\simeq
\pi_1(S_2)^{\mathrm{un},B}.
\]
\end{Proposition}

\begin{proof}
We give the proof for $S_2$; the proof for $S_1$ is the same. Set
\[
\Gamma:=\pi_1(S_2)\simeq \mathbb Z^2\rtimes \mathbb F_3\ .
\]
Let $x,y$ generate the normal subgroup $\mathbb Z^2$, and let
$g_1,g_2,g_3$ be free generators of $\mathbb F_3$. By
\eqref{e:monodromy-relations}, we have
$g_jxg_j^{-1}=x^{-1}$ 
and $g_jyg_j^{-1}=y^{-1}$.

By right exactness of Malcev completion \cite[Prop.~3.6]{hain-modular},
applied to
\[
1\rightarrow \mathbb Z^2
\rightarrow \Gamma
\rightarrow \mathbb F_3
\rightarrow 1
\]
we obtain a right exact sequence of pro-unipotent algebraic groups
\[
(\mathbb Z^2)^{\mathrm{un}}
\rightarrow
\Gamma^{\mathrm{un}}
\rightarrow
\mathbb F_3^{\mathrm{un}}
\rightarrow 1.
\]
Since
$(\mathbb Z^2)^{\mathrm{un}}\simeq \mathbb G_a^2$
this becomes
\begin{equation}\label{e:right-exact-malcev}
\mathbb G_a^2
\rightarrow
\Gamma^{\mathrm{un}}
\rightarrow
\mathbb F_3^{\mathrm{un}}
\rightarrow 1.
\end{equation}
It remains to show that the first morphism is trivial.

We use the following elementary observation. Let $U$ be a pro-unipotent group
over $\mathbb Q$, and let $a,h\in U(\mathbb Q)$ satisfy
\[
hah^{-1}=a^{-1}.
\]
Then $a=1$. Indeed, it is enough to check this after passing to every
finite-dimensional unipotent quotient of $U$. In such a quotient, write
\[
a=\exp(X)
\qquad
\text{and}
\qquad
h=\exp(H)
\]
with $X,U$ in $\mathrm{Lie}(U)$. The relation $hah^{-1}=a^{-1}$ gives
$\operatorname{Ad}(h)(X)=-X$. On the other hand
$\operatorname{Ad}(h)=\exp(\operatorname{ad}(H))$
is a unipotent operator. Hence all its eigenvalues are equal to $1$, so it
has no eigenvalue $-1$. Therefore $X=0$, and hence $a=1$.

We apply this observation to the canonical completion map
$\iota_\Gamma\colon
\Gamma\rightarrow \Gamma^{\mathrm{un}}(\mathbb Q)$.
Using the relations \eqref{e:monodromy-relations}, we obtain
\[
\iota_\Gamma(x)=1
\qquad
\text{and}
\qquad
\iota_\Gamma(y)=1.
\]
Thus the composite
\[
\mathbb Z^2
\hookrightarrow
\Gamma
\xrightarrow{\iota_\Gamma}
\Gamma^{\mathrm{un}}(\mathbb Q)
\]
is trivial. By the universal property of the unipotent completion of
$\mathbb Z^2$, we obtain that the induced morphism
$\mathbb G_a^2=(\mathbb Z^2)^{\mathrm{un}}
\rightarrow
\Gamma^{\mathrm{un}}$ is therefore trivial.

Hence the cokernel of the first morphism in \eqref{e:right-exact-malcev} is $\Gamma^{\mathrm{un}}$ itself. Since \eqref{e:right-exact-malcev} is right exact, it follows that
$\Gamma^{\mathrm{un}}
\simeq
\mathbb F_3^{\mathrm{un}}$.
\end{proof}


\appendix 

\section{On the fundamental groups of \texorpdfstring{$\mathbb R$}{R}-points} 
\label{sec:pi_Real}

We now discuss the analogues of the computations of Section~2 over $\mathbb{R}$. We begin with a few observations about $\Upsilon(\mathbb{R})$: 

\begin{enumerate} 

\item[$\bullet$]  From the proof of~\cite[Corollary 3.2.3]{VL}, over the real numbers, $\Upsilon(\mathbb{R})$ has 24 singularities which are all ordinary double points.

\item[$\bullet$] By looking at the defining equations of $\Upsilon(\mathbb{R})$, (viewed as a projective variety) over the real numbers we see that we must have that $U\not=0$, as otherwise, all variables must equal to zero. Thus we will view $\Upsilon(\mathbb{R})$ as a compact affine variety by putting $U=1$. The last equation becomes $A^2+B^2+C^2=1$, from this we can conclude that $\Upsilon(\mathbb{R})$ is compact, as it is bounded in $\mathbb{R}^{6}$.
\end{enumerate}

We use these observations to compute the  fundamental group of $\Upsilon(\mathbb{R})$ using a few key results.

The following Lemma shows that the real locus of $\Upsilon$ is path connected. This can be easily noticed since $\Upsilon(\mathbb{R})$ is an irreducible real variety and therefore connected, but we add a more detailed proof for the sake of completeness.

\begin{Lemma}\label{l:UpsilonR-pathconn}
The real locus $\Upsilon(\mathbb R)$ is path connected.
\end{Lemma}

\begin{proof}
First note that $\Upsilon(\mathbb R)$ has no points with $U=0$. Indeed, from the last equation in
\eqref{eq:cuboids} we get $A^2+X^2=0$, hence $A=X=0$, and then the remaining equations force
$B=C=Y=Z=0$, which is impossible in projective space. Thus $\Upsilon(\mathbb R)$ is contained in the affine chart $U=1$. On  $U=1$, we recall that the equations \eqref{eq:cuboids} become
\[
X^{2}=B^{2}+C^{2},\qquad
Y^{2}=C^{2}+A^{2},\qquad
Z^{2}=A^{2}+B^{2},\qquad
A^{2}+X^{2}=1,
\]
and using $X^{2}=B^{2}+C^{2}$ the last equation is equivalent to $A^{2}+B^{2}+C^{2}=1$.
Set
\[
S^{2}:=\{(A,B,C)\in\mathbb R^{3}:A^{2}+B^{2}+C^{2}=1\}.
\]

For each $\varepsilon=(\varepsilon_X,\varepsilon_Y,\varepsilon_Z)\in\{\pm1\}^{3}$, define $\Upsilon_{\varepsilon}(\mathbb R)$ as the following locus
\[
\Bigl\{(A,B,C,X,Y,Z)\in \Upsilon(\mathbb R):
X=\varepsilon_X\sqrt{B^{2}+C^{2}},\ 
Y=\varepsilon_Y\sqrt{C^{2}+A^{2}},\
Z=\varepsilon_Z\sqrt{A^{2}+B^{2}}
\Bigr\}.
\]
Then each $\Upsilon_{\varepsilon}(\mathbb R)$ is the image of the continuous map
\[
S^{2}\rightarrow \Upsilon(\mathbb R)\qquad
(A,B,C)\mapsto \bigl(A,B,C,\varepsilon_X\sqrt{B^{2}+C^{2}},\varepsilon_Y\sqrt{C^{2}+A^{2}},\varepsilon_Z\sqrt{A^{2}+B^{2}}\bigr),
\]
so $\Upsilon_{\varepsilon}(\mathbb R)$ is path connected, since it is homeomorphic to $S^{2}$.
Moreover,
\[
\Upsilon(\mathbb R)=\bigcup_{\varepsilon\in\{\pm1\}^{3}} \Upsilon_{\varepsilon}(\mathbb R).
\]

We now show that the union is path connected by checking that the sheets intersect in a connected pattern.
If $\varepsilon$ and $\varepsilon'$ differ only in the $X$--sign (i.e.\ $\varepsilon_Y=\varepsilon'_Y$ and
$\varepsilon_Z=\varepsilon'_Z$), then $\Upsilon_{\varepsilon}(\mathbb R)\cap \Upsilon_{\varepsilon'}(\mathbb R)\neq\emptyset$:
indeed, at $(A,B,C)=(1,0,0)\in S^{2}$ we have $X=0$, $Y=1$, $Z=1$, hence the point
$(1,0,0,0,\varepsilon_Y,\varepsilon_Z)$ lies in both sheets. Similarly, if $\varepsilon$ and $\varepsilon'$ differ only
in the $Y$--sign, they intersect at $(A,B,C)=(0,1,0)$, and if they differ only in the $Z$--sign, they intersect at
$(A,B,C)=(0,0,1)$.

Thus, whenever $\varepsilon$ and $\varepsilon'$ differ in exactly one coordinate, the intersection
$\Upsilon_{\varepsilon}(\mathbb R)\cap \Upsilon_{\varepsilon'}(\mathbb R)$ is nonempty. The graph on $\{\pm1\}^{3}$ in which two sign
vectors are adjacent if they differ in one coordinate is connected, as it consists of the $3$-dimensional cube. It follows that the union
$\bigcup_{\varepsilon}\Upsilon_{\varepsilon}(\mathbb R)=\Upsilon(\mathbb R)$ is path connected.
\end{proof}

\begin{Lemma}
    There exists a compact differential surface $M$ such that $$\pi_1(\Upsilon(\mathbb{R})) \simeq \pi_1(M) * \mathbb F_{24}$$ where $\mathbb F_{24}$ is the free group on $24$ generators. In particular, $\Upsilon(\mathbb{R})$ is not simply connected.    
\end{Lemma}

\begin{proof}
    We will construct $M$ in the following way. Let $p \in \Upsilon(\mathbb{R})$ be an ordinary double point. By Morse's lemma, there exists a neighborhood $p \in U \subset \Upsilon$ such that $U$ is homeomorphic to $$V(x^2+y^2-z^2) \cap \text{Ball}_{1}(0) \subset \mathbb{R}^3.$$ Therefore, $U \smallsetminus \{p\}$ is homeomorphic to an open cone from whom we have removed the origin, which itself is homeomorphic to the intersection of a ball with two disjoint cylinders. Therefore, by gluing two disks on either sides of $\Upsilon(\mathbb{R}) \smallsetminus \{p\}$, we have replaced $p$ by a pair of smooth points, and by performing this to every singular point of $\Upsilon(\mathbb{R})$, we constructed a compact differential surface $M$. Now, we can view $\Upsilon(\mathbb{R})$ as $M$ in which we identified $2 \cdot 24 = 48$ points. Yet, identifying a pair of points on a topological surface is the same as attaching a closed CW $1$-cell to it. Therefore, $\Upsilon(\mathbb{R})$ is homeomorphic to the wedge product of $M$ with $24$ circles, and so using Van-Kampen's Theorem, we conclude that $$\pi_1(\Upsilon(\mathbb{R})) = \pi_1(M) * \pi_1\left(\bigvee^{24} S^1\right) = \pi_1(M) * \mathbb F_{24}.$$ 
\end{proof}

\begin{Lemma}
    Let $\widetilde{\Upsilon}(\mathbb{R})$ be the minimal resolution of singularities of $\Upsilon(\mathbb{R})$ together with the resolution map $\pi\colon \widetilde{\Upsilon}\rightarrow \Upsilon$. Then $\widetilde{\Upsilon}(\mathbb{R})$ is homeomorphic to the connected sum of $M$ with $\Sigma_{24}$, where $\Sigma_{24}$ is the unique orientable surface of genus $24$. 
\end{Lemma}

\begin{proof}
    Since the only singularities of $\Upsilon(\mathbb{R})$ are ordinary double points, then $\widetilde{\Upsilon}(\mathbb{R})$ is constructed by blowing up each singularity of $\Upsilon(\mathbb{R})$ once. We also note that $\widetilde{\Upsilon}(\mathbb{R})$ is path connected since $\Upsilon$ is irreducible, and therefore so is $\widetilde{\Upsilon}(\mathbb{R})$. Let $p \in \Upsilon(\mathbb{R})$ be an ordinary double point, and so from Morse's lemma (see, for example, Lemma 2.2 in~\cite{Milnor}) there exists a neighborhood $p \in U \subset \Upsilon(\mathbb{R})$ such that $U$ is homeomorphic to 
    $$V(x^2+y^2-z^2) \cap \text{Ball}_{1}(0) \subset \mathbb{R}^3.$$
    Now, the blow up $V(x^2+y^2-z^2)$ at the origin is homeomorphic to the $\mathbb{R} \times \mathbb{P}_{\mathbb{R}}^1$ as the exceptional divisor is isomorphic to $\mathbb{P}_{\mathbb{R}}^1$, and blowing up $\Upsilon(\mathbb{R})$ at $p$ is a homeomorphism outside $U$. But locally outside $p$ we have that $M \smallsetminus \{p\}$ is homeomorphic to $\widetilde{\Upsilon}(\mathbb{R}) \smallsetminus \pi^{-1}(p)$. Thus, topologically, the blow up of $\Upsilon(\mathbb{R})$ at $p$ replaces the CW $1$-cell which connects the two points  corresponding to $x$ in $M$ by a copy of $\mathbb{R} \times \mathbb{P}^1_{\mathbb{R}} = \mathbb{R} \times S^1$. Therefore, it corresponds to the connect sum of $M$ with a torus. By performing this over all $24$ singular points, the result follows. 
\end{proof}

\begin{Lemma}
    $\widetilde{\Upsilon}(\mathbb{R})$ is non orientable. 
\end{Lemma}

\begin{proof}
    Since $\widetilde{\Upsilon}(\mathbb{R})$ is the resolution of $\Upsilon(\mathbb{R})$, it is a compact differentiable surface. In addition, we can write the resolution as a series of blow ups at its ordinary double points $\Upsilon_N \to \Upsilon_{N-1} \to \cdots \to \Upsilon_1 \to \Upsilon_0 = \Upsilon$. We will prove that for every $i$ we have that $H^2(\Upsilon_i, \mathbb{Z})$ is not torsion free, and so the result follows from the cohomological description of orientability. Let $x_i \in \Upsilon_i$ be the point we blow up under the map $\pi_i \colon \Upsilon_{i+1} \to \Upsilon_i$ and let $E_{i+1} = \pi^{-1}(x_i)$. Then we have a long exact sequence 
    \begin{equation*}
        \cdots \to H^1(E_{i+1}, \mathbb{Z}) \to H^2(\Upsilon_i, \mathbb{Z}) \to H^2(\Upsilon_{i+1}, \mathbb{Z}) \oplus H^2(\{x_i\}, \mathbb{Z}) \to H^2(E_{i+1}, \mathbb{Z}) \to \cdots. 
    \end{equation*}
    Yet we have that $H^2(\{x_i\}, \mathbb{Z})=0$ and that $E_{i+1}$ is homeomorphic to $\mathbb{P}_{\mathbb{R}}^1$, and therefore $H^2(E_{i+1}, \mathbb{Z})=0$ and $H^1(E_{i+1}, \mathbb{Z})=\mathbb{Z}$. Thus the long exact sequence becomes
    $$ \cdots \to \mathbb{Z} \to H^2(\Upsilon_i, \mathbb{Z}) \to H^2(\Upsilon_{i+1}, \mathbb{Z}) \to 0.$$ Therefore the map $H^2(\Upsilon_i, \mathbb{Z}) \to H^2(\Upsilon_{i+1}, \mathbb{Z})$ is surjective. If $\Upsilon_i$ is non orientable then so is $\Upsilon_{i+1}$, and if $\Upsilon_i$ is orientable we have that $H^2(\Upsilon_i, \mathbb{Z})=\mathbb{Z}$,  thus $H^2(\Upsilon_{i+1}, \mathbb{Z})$ is the non-zero cokernel of a map $\mathbb{Z} \to \mathbb{Z}$, which can never be torsion free. Therefore $\Upsilon_{i+1}$ is never orientable for every $i$ and the result follows. 
\end{proof}

\begin{Corollary}
 \textup{$\widetilde{\Upsilon}(\mathbb{R})$ is homeomorphic to $N_k$ for some $k$.}
\end{Corollary}

\begin{proof} This follows from the theorem of classification of surfaces.
\end{proof}

\begin{Lemma}
   The surface $M$ is non-orientable and homeomorphic to $N_{k-48}$.
\end{Lemma}

\begin{proof}
    Since $\widetilde{\Upsilon}(\mathbb{R})$ is non orientable and homeomorphic to the connected sum of $M$ with $\Sigma_{24}$, then if $M$ would have been orientable, then so would $\widetilde{\Upsilon}(\mathbb{R})$. Thus $M$ is homeomorphic to $N_r$ for some $r$. Thus the connected sum of $N_{r}$ and $\Sigma_{24}$ is homeomorphic $N_{48+r}$, but is homeomorphic to $\widetilde{\Upsilon}(\mathbb{R})$, which itself is homeomorphic to $N_{k}$. Thus $k=48+r$ and the result follows.
\end{proof}

\begin{Lemma} \label{prop:k}  Let $\chi(\Upsilon(\mathbb{R}))$ be the topological Euler characteristic of $\Upsilon(\mathbb{R})$. Then we have  $k=26 - \chi(\Upsilon(\mathbb{R}))$.
\end{Lemma}

\begin{proof} Since $\Upsilon(\mathbb{R})$ is constructed by attaching $24$ CW $1$-cells to $M$, we get that $\chi(\Upsilon(\mathbb{R}))=\chi(M) -24$. But since $M$ is homeomorphic to $N_{k-48}$, we have that $\chi(M)=2-(k-48)=50-k$. Thus $\chi(\Upsilon(\mathbb{R})) = (50-k)-24=26-k$ which gives us that $k=26 - \chi(\Upsilon(\mathbb{R}))$. 
\end{proof}

We thus immediately obtain the following corollary.

\begin{Corollary} The fundamental group  $\pi_1(\Upsilon(\mathbb{R}))$ is isomorphic to the free product of $\mathbb F_{24}$ and $\pi_1(N_{k-48})$. 
\end{Corollary}

We describe a way to compute the Euler characteristics $\chi(\Upsilon(\mathbb{R}))$ and $\chi(V(\mathbb{R}))$ in terms of Milnor numbers of polynomials. We recall a result by Bruce from~\cite{bruce}. 
 
\begin{Proposition}[Proposition 7 in~\cite{bruce}]\label{prop:bruce}
    Let $f_1, \dots, f_r \colon \mathbb{R}^n \to \mathbb{R}$ be polynomials of degree $\leq d$ and suppose $X=V(f_1, \dots, f_r)$ is compact in the Euclidean topology of $\mathbb{R}^n$. Then  $$\chi(X) = \frac{(-1)^n- \mu(H)}{2}$$ where $\chi(X)$ is the topological Euler characteristic of $X$ and $H \colon \mathbb{R}^{n+1} \to \mathbb{R}$ is defined to be 
    \begin{equation*}
        H(x_1,\dots,x_n,y) = \left(\sum_{i=1}^r y^{d+1} f\left(\frac{x_1}{y}, \dots, \frac{x_n}{y}\right)\right) - y^{2d+4} - x_1^{2d+4} - \cdots -x_n^{2d+4},
    \end{equation*}
    and by $\mu(H)$ we mean the real Milnor number of $H$ at the origin, i.e., the dimension of $$
    \frac{\mathbb{R} [[y,x_1, \dots, x_n]]}{\langle \frac{\partial H}{\partial y}, \frac{\partial H}{\partial x_1}, \dots, \frac{\partial H}{\partial x_n} \rangle}
    $$ as a real vector space, where $\mathbb{R} [[x_1, \dots, x_n]]$ is the ring of real analytic functions in variables $x_1, \dots, x_n$. 
\end{Proposition}

In order to compute $\chi(\Upsilon(\mathbb{R}))$ and $\chi(V(\mathbb{R}))$, we can apply Bruce's formula to equations (\ref{eq:cuboids}) and (\ref{eq:face_cuboids}) respectively. 

Combining Proposition~\ref{prop:k} and Proposition~\ref{prop:bruce} we obtain :
\begin{equation} \label{eq:k} k = 26 - \chi(\Upsilon(\mathbb{R})) = 26 - \frac{(-1)^n- \mu(H_{\Upsilon})}{2} .
\end{equation}

\begin{Proposition}  $\pi_{1}(\Upsilon(\mathbb{R})) = \pi_1(N_{k-48}) * \mathbb F_{24}\;\;\text{and}\;
\pi_{1}(\widetilde{\Upsilon}(\mathbb{R})) = \pi_1(N_k)$, where 
\begin{equation} \label{eq:k-final} k = 26 - \chi(\Upsilon(\mathbb{R})) = 26 - \frac{(-1)^n- \mu(H_{\Upsilon})}{2} \in \mathbb{N}. 
\end{equation}
where $H_{\Upsilon}\in \mathbb{R}[A,B,C,X,Y,Z,D]$ is the polynomial defined as
\begin{multline*}
H_{\Upsilon}(A,B,C,X,Y,Z,D)= D^2((A^2 + B^2 - Z^2)^2 + (B^2 + C^2 - X^2)^2 
\\ + (C^2 + A^2 - Y^2)^2 + (A^2 + B^2 + C^2 - D^2)^2) - A^8 - B^8 -C^8 -X^8 - Y^8 -Z^8 -D^8.
\end{multline*}
\end{Proposition}

Concerning $V(\mathbb{R})$ and $\widetilde{V}(\mathbb{R})$, the same proof works. Let us summarize it below.

\begin{Proposition}\label{prop:k-prime}  $\pi_{1}(V(\mathbb{R})) = \pi_1(N_{k'-32}) * \mathbb F_{16}$ and $\pi_{1}(\widetilde{V}(\mathbb{R})) = \pi_1(N_{k'})$, where $\mathbb F_{16}$ is the free group on $16$ generators and 

\begin{equation}\label{eq:k'} k'=18 - \chi(V(\mathbb{R})) = 18 -  \frac{(-1)^{n} - \mu(H_{V})}{2}
\end{equation}

\begin{equation*}H_{V}(A,B,C,X,Y,D)= H_{\Upsilon}(A,B,C,X,Y,0,D) \in \mathbb{R}[A,B,C,X,Y,D]
\end{equation*}

\end{Proposition}

\begin{proof} The proof is similar to the proof for $\Upsilon$, except that $V(\mathbb{R})$ has 16 singularities which are ordinary double points. There exists, as in proposition 3.1, a compact differential surface $M'$ such that $\pi_{1}(V(\mathbb{R}))\simeq \pi_{1}(M') \ast \mathbb F_{16}$.

Now, $\widetilde{V}(\mathbb{R})$ is homeomorphic to the connected sum of $M'$ with $\Sigma_{16}$. Also, $\widetilde{V}(\mathbb{R})$ is non orientable as in Proposition 3.4, thus homeomorphic to $N_{k'}$ for some $k'$. Moreover, $M$ is also homeomorphic to some $N_{r'}$ for some $r'$, and thus $\widetilde{V}(\mathbb{R})$ is homeomorphic to $N_{32+r'}$, whence $k'=32+r'$. Whence $M'$ is homeomorphic to $N_{k'-32}$. Now since $V(\mathbb{R})$ is constructed by attaching $16$ CW $1$-cells to $M'$, we have $$\chi(V(\mathbb{R}))=\chi(M') - 16 = 2 - (k'-32) - 16 = 18 - k'.$$
\end{proof}

\end{document}